\documentclass[11pt]{amsart}
\usepackage[margin=35mm]{geometry}
\usepackage{amsmath,amssymb}
\usepackage{amsthm}

\newtheorem{thm}{Theorem}[section]
\newtheorem{prop}{Proposition}[section]
\newtheorem{lem}{Lemma}[section]
\newtheorem{cor}{Corollary}[section]

\newtheorem{ex}{Example}[section]
\newtheorem{rem}{Remark}[section]
 
\begin{document}

\title{On the shadowing and limit shadowing properties}
\author{Noriaki Kawaguchi}
\address{Graduate School of Mathematical Sciences, The University of Tokyo, 3-8-1 Komaba Meguro, Tokyo 153-8914, Japan}
\email{knoriaki@ms.u-tokyo.ac.jp}

\maketitle

\markboth{NORIAKI KAWAGUCHI}{On the shadowing and limit shadowing properties}

\begin{abstract}
We study the relation between the shadowing property and the limit shadowing property. We prove that if a continuous self-map $f$ of a compact metric space has the limit shadowing property, then the restriction of $f$ to the non-wandering set satisfies the shadowing property. As an application, we prove the equivalence of the two shadowing properties for equicontinuous maps.  
\end{abstract}

\section{Introduction}

Shadowing has been the subject of much interest in the qualitative study of dynamical systems \cite{AH, P1}, and various shadowing properties have been defined in the course of such studies so far. The limit shadowing property introduced in \cite{ENP} is one of the variants of the shadowing property, which focuses on the possibility of asymptotic shadowing of pseudo orbits whose one-step errors are converging to zero, and it is a subject of ongoing research, see \cite{BGOR, C, CK, KKO, KO, L, MO, O, P2, S}.

In \cite{P2, S}, various limit shadowing properties are examined in relation to the notion of hyperbolicity and stability. The set of $\Omega$-stable diffeomorphisms of a smooth closed manifold is characterized as the $C^1$-interior of the set of diffeomorphisms satisfying the limit shadowing property \cite{P2}. The s-limit shadowing property is a stronger property than the limit shadowing property defined in combination with the shadowing property. It has been proved to be a $C^0$-dense property on the space of continuous self-maps (resp. continuous surjections) of a compact topological manifold \cite{MO}. The two-sided limit shadowing property of homeomorphisms is a bilateral version of the limit shadowing property, and its consequences are given, for example, in \cite{C, CK, L, O}. It is, in fact, a much stronger property than the limit shadowing property and a sufficient condition for chaos \cite{L,O}. In \cite{BGOR}, the limit shadowing property is exploited to characterize the $\omega$-limit sets in topologically hyperbolic systems in terms of the notions such as internal chain transitivity. Furthermore, in \cite{KKO, KO}, the limit shadowing property is studied together with other types of shadowing properties like the average shadowing property (ASP) and the asymptotic average shadowing property (AASP). AASP is a stochastic version of the limit shadowing property, yet there is a large class of homeomorphisms satisfying both of the shadowing and limit shadowing properties but not AASP \cite{KO}.

First, we give the formal definitions of the standard and limit shadowing properties. Throughout this paper, we deal with a continuous map $f\colon X\to X$ on a compact metric space $(X,d)$. For $\delta>0$, a sequence $(x_i)_{i\ge0}$ of points in $X$ is a {\em $\delta$-pseudo orbit} of $f$ if $d(f(x_i),x_{i+1})\le\delta$ for all $i\ge0$.  Then, for given $\epsilon>0$, a $\delta$-pseudo orbit  $(x_i)_{i\ge0}$ is said to be {\em $\epsilon$-shadowed} by $x\in X$ if $d(x_i,f^i(x))\leq \epsilon$ for all $i\ge 0$. We say that $f$ has the {\em shadowing property} if for any $\epsilon>0$, there is $\delta>0$ such that every $\delta$-pseudo orbit of $f$ is $\epsilon$-shadowed by some point of $X$. A sequence $(x_i)_{i\ge0}$ of points in $X$ is a {\em limit pseudo orbit} of $f$ if $\lim_{i\to\infty}d(f(x_i),x_{i+1})=0$. We say that $f$ has the {\em limit shadowing property} if for any limit pseudo orbit  $(x_i)_{i\ge0}$ of $f$, there is $y\in X$ such that $\lim_{i\to\infty}d(x_i,f^i(y))=0$, and such $y$ is called a {\em limit shadowing point} of $(x_i)_{i\ge0}$.

An interesting problem is to study the relation between the (standard) shadowing property and the limit shadowing property. In \cite{P1}, an example of circle homeomorphism (or diffeomorphism) satisfying the limit shadowing property but not the shadowing property is given. By taking the direct product of such diffeomorphisms, we obtain diffeomorphisms of arbitrary dimensional torus with the limit shadowing property but without the shadowing property. We also give such a simple example below.

\begin{ex}
\normalfont
Let $X=\{0,1,2\}\cup\{s_n\colon n\in\mathbb{Z}\}\cup\{t_n\colon n\in\mathbb{Z}\}$, where $(s_n)_{n\in\mathbb{Z}}$ and $(t_n)_{n\in\mathbb{Z}}$ are sequences of real numbers satisfying the following properties
\begin{itemize}
\item[(1)] $s_n<s_{n+1}$ and $t_n<t_{n+1}$ for all $n\in\mathbb{Z}$,
\item[(2)] $\lim_{n\to-\infty}s_n=0$, $\lim_{n\to+\infty}s_n=\lim_{n\to-\infty}t_n=1$, and $\lim_{n\to+\infty}t_n=2$.
\end{itemize}
Then, $X$ is a compact subset of $[0,2]$. Define $f\colon X\to X$ by
\begin{itemize}
\item[(3)] $f(y)=y$ for $y\in\{0,1,2\}$,
\item[(4)] $f(s_n)=s_{n+1}$ and $f(t_n)=t_{n+1}$ for all $n\in\mathbb{Z}$.
\end{itemize}
Then, $f$ is an expansive homeomorphism. For any limit pseudo orbit $(x_i)_{i\ge0}$ of $f$, it is easy to see that $\lim_{i\to\infty}x_i=y$ for some $y\in\{0,1,2\}$, and this implies that $\lim_{i\to\infty}d(x_i,f^i(y))=0$ since $y$ is a fixed point. Hence, $f$ has the limit shadowing property, and it is also easy to see that $f$ does not have the shadowing property. Let $g=f^\mathbb{N}\colon X^\mathbb{N}\to X^\mathbb{N}$ be the direct product of countably many copies of $f$. For any continuous map $h\colon Y\to Y$ of a compact metric space, the direct product $h^\mathbb{N}\colon Y^\mathbb{N}\to Y^\mathbb{N}$ satisfies the (limit) shadowing property iff so does $h$. Hence, $f$ is a homeomorphism of a Cantor space satisfying the limit shadowing property but not the shadowing property. We have $\Omega(g)=\{0,1,2\}^\mathbb{N}=Fix(g)$, a Cantor space. Since $Fix(g)$ is uncountable, $g$ cannot be expansive.
\end{ex}

Also in \cite{P1}, it is proved that for circle homeomorphisms, the shadowing property always implies the limit shadowing property, and the same implication holds true for c-expansive maps including expansive homeomorphisms (see \cite{BGO, BGOR, LS}). It is rather difficult to construct a continuous map satisfying the shadowing property but not the limit shadowing property, but in \cite{GOP}, such an example is given, while the equivalence of the two shadowing properties is proved for a certain class of interval maps.

The above facts may indicate that the limit shadowing property is weaker than the shadowing property at least intuitively. However, in \cite{KKO}, as a (partial) converse, it is proved that if a continuous map with the limit shadowing property is chain transitive, then it satisfies the shadowing property. As the main result of this paper, through a generalization of the result for the chain recurrent case (Lemma 2.1), we prove a basic relation between the two shadowing properties. To state the result, we give some definitions and notations.

Given a continuous map $f\colon X\to X$, a finite sequence of points $(x_i)_{i=0}^{k}$ in $X$  (where $k$ is a positive integer) is called a {\em $\delta$-chain} of $f$ if $d(f(x_i),x_{i+1})\le\delta$ for every $0\le i\le k-1$. We say that $f$ is {\em chain transitive} if for any $x,y\in X$ and $\delta>0$, there is a $\delta$-chain $(x_i)_{i=0}^{k}$ of $f$ such that $x_0=x$ and $x_k=y$. A $\delta$-chain $(x_i)_{i=0}^{k}$ of $f$ is said to be a {\em $\delta$-cycle} of $f$ if $x_0=x_k$, and a point $x\in X$ is a {\em chain recurrent point} for $f$ if for any $\delta>0$, there is a $\delta$-cycle $(x_i)_{i=0}^{k}$ of $f$ with $x_0=x_k=x$. We denote by $CR(f)$ the set of chain recurrent points for $f$. A point $x\in X$ is said to be {\em minimal} if the restriction of $f$ to the orbit closure $\overline{O_f(x)}=\overline{\{f^n(x)\colon n\ge0\}}$ is minimal, and {\em non-wandering} if for every neighborhood $U$ of $x$, we have $f^n(U)\cap U\ne\emptyset$ for some $n>0$. We denote by $M(f)$ (resp. $\Omega(f)$) the set of minimal (resp. non-wandering) points for $f$. Note that $M(f)\subset\Omega(f)\subset CR(f)$.

Now, we state the theorem.

\begin{thm}
Let $f\colon X\to X$ be a continuous map with the limit shadowing property. Then, $CR(f)=\Omega(f)=\overline{M(f)}$, and $f|_{\Omega(f)}\colon\Omega(f)\to\Omega(f)$ satisfies the shadowing property.
\end{thm}

This theorem enables us to apply the developed theory of the shadowing property to continuous maps enjoying the limit shadowing property. As an example, we obtain the following corollary (cf. \cite{Ka, LO}).

\begin{cor}
Let $f\colon X\to X$ be a continuous map with the limit shadowing property. Then, for every $x\in\Omega(f)$ and every $\epsilon>0$, there exists $y\in\Omega(f)$ with $d(x,y)<\epsilon$ such that $y$ is a periodic point for $f$, or $f|_{\overline{O_f(y)}}\colon\overline{O_f(y)}\to\overline{O_f(y)}$ is topologically conjugate to an odometer.
\end{cor}

To state another corollary of Theorem 1.1, we briefly recall the definition of the so-called {\em thick shadowing property}. Following the notation in \cite{O2}, we define two families of subsets in $\mathbb{N}_0=\mathbb{N}\cup\{0\}$ by
\begin{eqnarray*}
&&\mathcal{D}=\{A\subset\mathbb{N}_0\colon\lim_{n\to\infty}\,\frac{1}{n}\,|A\cap\{0,1,\dots,n-1\}|=1\},\\
&&\mathcal{F}_t=\{B\subset\mathbb{N}_0\colon\forall n\in\mathbb{N}_0\;\exists j\in\mathbb{N}_0\;\;{\rm s.t.}\;\;\{j,j+1,\dots,j+n\}\subset B\}.
\end{eqnarray*}
Each member of $\mathcal{F}_t$ is called a {\em thick set}. Note that $\mathcal{D}\subset\mathcal{F}_t$. Then, given a continuous map $f\colon X\to X$, a sequence $(x_i)_{i\ge0}$ of points in $X$ is said to be an {\em ergodic $\delta$-pseudo orbit} of $f$ if
\begin{equation*}
\{i\ge0\colon d(f(x_i),x_{i+1})\le\delta\}\in\mathcal{D},
\end{equation*}
and $(x_i)_{i\ge0}$ is said to be {\em $\epsilon$-shadowed on a thick set} by $x\in X$ if
\begin{equation*}
\{i\ge0\colon d(x_i,f^i(x))\le\epsilon\}\in\mathcal{F}_t.
\end{equation*}
We say that a continuous map $f\colon X\to X$ has the {\em thick shadowing property} if for any $\epsilon>0$, there exists $\delta>0$ such that every ergodic $\delta$-pseudo orbit of $f$ is   $\epsilon$-shadowed on a thick set by some point of $X$.

The thick shadowing property was introduced in \cite{DH} as a shadowing property defined by restricting the time at which pseudo orbits have small errors and the time at which true orbits closely shadow them to proper subclasses of the index set (e.g. ergodic shadowing property, $\underline{d}$ or $\overline{d}$-shadowing property). In \cite{BMR}, the thick shadowing property and several similar properties were proved to be equivalent to the shadowing property under the assumption of chain transitivity. The study was extended in \cite{O2}, and a characterization of the thick shadowing property was given. According to \cite[Theorem 4.5]{O2}, a continuous map $f\colon X\to X$ has the thick shadowing property iff $CR(f)=\Omega(f)=\overline{R(f)}$ and $f|_{\Omega(f)}$ has the shadowing property, where $R(f)$ denotes the set of recurrent points for $f$, i.e., $R(f)=\{x\in X\colon x\in\omega(x)\}$. Note that we have $M(f)\subset R(f)$. Thus, by Theorem 1.1, we obtain the following corollary.

\begin{cor}
For any continuous map $f\colon X\to X$, if $f$ has the limit shadowing property, then $f$ has the thick shadowing property.
\end{cor}

\begin{rem}
\normalfont
The converse of Corollary 1.2 does not hold in general. In fact, as a corollary of \cite[Theorem 4.5]{O2}, if a continuous map $f$ has the shadowing property, then $f$ has the thick shadowing property. On the other hand, in \cite{GOP}, an example of continuous map with the shadowing property but without the limit shadowing property is given, so such $f$ shows that the thick shadowing property does not necessarily imply the limit shadowing property.
\end{rem}

The next result of this paper concerns the notion of equicontinuity. A map $f\colon X\to X$ is said to be {\em equicontinuous} if for any $\epsilon>0$, there is $\delta>0$ such that $d(x,y)\le\delta$ implies \[\sup_{n\ge0}d(f^n(x),f^n(y))\le\epsilon\] for all $x,y\in X$. It is known that if an equicontinuous map $f$ is surjective, then $f$ is a homeomorphism, and $f^{-1}$ is also equicontinuous (cf. \cite{AG, Ma}). When $f\colon X\to X$ is an equicontinuous map (or an equicontinuous homeomorphism), there is a metric $D$ equivalent to $d$ on $X$ such that $D(f(x),f(y))\le D(x,y)$ (furthermore, $f$ is an isometry with respect to $D$) for all $x,y\in X$. In fact, $D\colon X\times X\to[0,\infty)$ defined by \[D(x,y)=\sup_{n\ge0}d(f^n(x),f^n(y))\] (or $\sup_{n\in\mathbb{Z}}d(f^n(x),f^n(y)))$ is such a metric. Every equicontinuous homeomorphism $f\colon X\to X$ is known to satisfy $X=M(f)$ (see \cite{Ma}). Under the assumption of equicontinuity, the shadowing property is closely tied to the notion of {\em chain continuity} introduced in \cite{A}. Given a continuous map $f\colon X\to X$, a point $x\in X$ is said to be a {\em chain continuity point} for $f$ if for any $\epsilon>0$, there is $\delta>0$ such that every $\delta$-pseudo orbit $(x_i)_{i\ge0}$ of $f$ with $x_0=x$ is $\epsilon$-shadowed by $x$ itself. An equicontinuous map $f\colon X\to X$ satisfies the shadowing property iff every $x\in X$ is a chain continuity point for $f$ (see \cite{A}).

As an application of Theorem 1.1, we prove the equivalence of the two shadowing properties for equicontinuous maps (including equicontinuous homeomorphisms).

\begin{thm}
Let $f\colon X\to X$ be an equicontinuous map. Then, the following three properties are equivalent
\begin{itemize}
\item[(1)] $f$ has the limit shadowing property,
\item[(2)] $f$ has the shadowing property,
\item[(3)] ${\rm dim}\,\Omega(f)=0$, or equivalently, $\Omega(f)$ is totally disconnected.
\end{itemize} 
\end{thm}

This theorem generalizes a result of \cite{BGO} proving that every equicontinuous homeomorphism of a totally disconnected space (e.g. odometer) satisfies the shadowing property and the limit shadowing property.

We proceed to present the next corollary. For a continuous surjection $f\colon X\to X$, let $X_f$ denote the set of bi-infinite sequences of points $(x_i)_{i\in\mathbb{Z}}\in X^\mathbb{Z}$ such that $f(x_i)=x_{i+1}$ for every $i\in\mathbb{Z}$. Then, $f$ is said to be {\em c-expansive} when there exists $e>0$ such that for any $x=(x_i)_{i\in\mathbb{Z}}, y=(y_i)_{i\in\mathbb{Z}}\in X_f$, if $d(x_i,y_i)\le e$ for all $i\in\mathbb{Z}$, then $x=y$ (see \cite{AH} for details). By \cite[Theorem 3.4]{BGOR} together with \cite[Theorem 3.7]{BGO} and by Theorem 1.2, we know that for c-expansive or equicontinuous maps, the shadowing property implies the limit shadowing property. For convenience, we state this fact as the next proposition. 

\begin{prop}
Let $f\colon X\to X$ be a c-expansive or equicontinuous map. If $f$ has the shadowing property, then $f$ has the limit shadowing property.
\end{prop}

By using Proposition 1.1 with Theorem 1.1 and Lemma 3.4 in Section 3, we obtain the following corollary.

\begin{cor}
Let $f\colon X\to X$ be a c-expansive or equicontinuous map.
\begin{itemize}
\item[(1)] If $f$ has the limit shadowing property, then $f|_{\Omega(f)}$ also has the limit shadowing property.
\item[(2)] $f$ has the limit shadowing property if and only if $f$ has the thick shadowing property.
\end{itemize}
\end{cor}

\begin{rem}
\normalfont
The author does not know whether the implication as Corollary 1.3 (1) holds true for every continuous map. By Corollary 1.2, the limit shadowing property always implies the thick shadowing property, but as mentioned in Remark 1.1, the converse does not hold in general. Corollary 1.3 (2) states that the converse also holds in the class of c-expansive or equicontinuous maps.
\end{rem}

\begin{rem}
\normalfont
We know that every equicontinuous homeomorphism with the shadowing property satisfies the h-shadowing property, and this implies the s-limit shadowing property \cite[Theorems 3.7 and 6.1]{BGO}. Indeed, the s-limit shadowing property implies both of the shadowing and limit shadowing properties (see \cite[Theorem 3.7]{BGO}). Therefore, by Lemma 3.1 in Section 3, the shadowing, limit shadowing, and s-limit shadowing properties are equivalent for equicontinuous homeomorphisms. Then, it is natural to expect that the equivalence holds true for general equicontinuous maps. However, this is not the case. We show it in Section 4 by giving an example of equicontinuous map with the shadowing property but without the s-limit shadowing property. The precise definition of the s-limit shadowing property will be given at the beginning of Section 4.
\end{rem}

This paper consists of four sections. We prove Theorem 1.1 in Section 2. Theorem 1.2 and Corollary 1.3 are proved in Section 3. In Section 4, we give an example of equicontinuous map showing that the s-limit shadowing property is not necessarily equivalent to the three properties in Theorem 1.2.   

\section{Proof of Theorem 1.1}

In this section, we prove Theorem 1.1. For the purpose, we need a lemma which generalizes \cite[Theorem 7.3] {KKO}. For a continuous map $f\colon X\to X$ and $S\subset X$, we say that $f$ has the {\em shadowing property around $S$} if for any $\epsilon>0$, there exists $\delta>0$ such that every $\delta$-pseudo orbit of $f$ contained in $S$ is $\epsilon$-shadowed by some point of $X$, and we say that $f$ has the {\em limit shadowing property around $S$} if every limit pseudo orbit of $f$ contained in $S$ has a limit shadowing point in $X$.

\begin{lem}
Let $f\colon X\to X$ be a continuous map and let $S\subset X$ be a compact $f$-invariant subset such that $CR(f|_S)=S$. If $f$ has the limit shadowing property around $S$, then $f$ has the shadowing property around $S$.
\end{lem}

\begin{proof}
For $\delta>0$, we decompose $S$ into a disjoint union of $\delta$-chain components. Precisely, for any $\delta>0$, we define a relation $\sim_\delta$ on $S$ as follows. For $x,y\in S$, $x\sim_\delta y$ iff there are two $\delta$-chains  $(x_i)_{i=0}^k\subset S$ and $(y_i)_{i=0}^l\subset S$ of $f$ such that $x_0=y_l=x$ and $x_k=y_0=y$. It is clear from the definition that $\sim_\delta$ is symmetric and transitive. Then, the assumption $CR(f|_S)=S$ guarantees the following properties
\begin{itemize}
\item[(1)] $x\sim_\delta x$ for every $x\in S$,
\item[(2)] $x\sim_\delta y$ for all $x,y\in S$ with $d(x,y)<\delta$,
\item[(3)] $x\sim_\delta f(x)$ for every $x\in S$.
\end{itemize} 
By (1),  $\sim_\delta$ is an equivalence relation on $S$. Moreover, by (2) and (3), every equivalence class $C$ with respect to $\sim_\delta$ is clopen in $S$ and $f$-invariant, i.e., $f(C)\subset C$. Each equivalence class is called a $\delta$-chain component, and therefore $S$ is decomposed into finitely many $\delta$-chain components. We denote by ${\mathcal C}(\delta)$ the (finite) set of all $\delta$-chain components. Given $\delta_1<\delta_2$, note that $x\sim_{\delta_1}y$ implies $x\sim_{\delta_2}y$ for all $x,y\in S$. Hence, for every $C\in{\mathcal C}(\delta_1)$, there is $D\in{\mathcal C}(\delta_2)$ such that $C\subset D$, and so for every $D\in{\mathcal C}(\delta_2)$, putting
\begin{equation*}
\mathcal{S}=\{C\in{\mathcal C}(\delta_1)\colon D\cap C\neq\emptyset\},
\end{equation*}
we have $D=\bigsqcup_{C\in\mathcal{S}}C$. In other words, every $\delta_2$-chain component is a disjoint union of some $\delta_1$-chain components.

Now, assume that $f$ does not have the shadowing property around $S$. Then, there is $\epsilon>0$ such that for every integer $n\ge1$, we can take an $n^{-1}$-chain $\gamma_n=(x_i^{(n)})_{i=0}^{l_n}\subset S$ of $f$ which is not $\epsilon$-shadowed by any point of $X$, meaning that for every $y\in X$, there is $0\le i \le l_n$ such that $d(x_i^{(n)},f^i(y))>\epsilon$. By using these chains, we shall construct a limit pseudo orbit in $S$ which has no limit shadowing point. For the purpose, we consider the $k^{-1}$-chain decomposition ${\mathcal C}(k^{-1})$ of $S$ for every integer $k\ge1$. Here, note that if $n^{-1}<\delta$, then $\gamma_n$ must be contained in some $C\in{\mathcal C}(\delta)$. This is because if $x\in C\in{\mathcal C}(\delta)$ and $d(f(x),y)\le n^{-1}<\delta$ for some $y\in S$, then we have $x\sim_\delta f(x)$ by (3), $f(x)\sim_\delta y$ by (2), and so $x\sim_\delta y$, implying $y\in C$. From this, we can take a $1$-chain component $C_1\in{\mathcal C}(1)$ and a sequence of integers $1\le n_{1,1}<n_{1,2}<\cdots$ such that $\gamma_{n_{1,j}}\subset C_1$ for every $j\ge1$. Then, since $C_1$ is a disjoint union of some elements of ${\mathcal C}(2^{-1})$, we can choose a $2^{-1}$-chain component $C_2\in{\mathcal C}(2^{-1})$ and a subsequence $1\le n_{2,1}<n_{2,2}<\cdots$ of $(n_{1,j})_{j\ge1}$ (i.e., $n_{2,j}=n_{1,\phi(j)}$ for some increasing function $\phi\colon\mathbb{N}\to\mathbb{N}$) such that $C_2\subset C_1$ and $\gamma_{n_{2,j}}\subset C_2$ for every $j\ge1$. Proceeding inductively, for every $k\ge1$, we obtain a $k^{-1}$-chain component $C_k\in{\mathcal C}(k^{-1})$ and  an increasing sequence of integers $(n_{k,j})_{j\ge1}$ with the following properties
\begin{itemize}
\item $\gamma_{n_{k,j}}\subset C_k$ for every $j\ge1$,
\item $C_{k+1}\subset C_k$,
\item $(n_{k+1,j})_{j\ge1}$ is a subsequence of $(n_{k,j})_{j\ge1}$.
\end{itemize} 
Then, putting $n_k=n_{k,k}$, we have $\gamma_{n_k}\subset C_k$ for every $k\ge1$. For each $k\ge1$, since $\gamma_{n_k}\cup\gamma_{n_{k+1}}\subset C_k$, we can take a $k^{-1}$-chain $\beta_k=(y_i)_{i=0}^{m_k}\subset S$ of $f$ such that $y_0=x_{l_{n_k}}^{(n_k)}$ and $y_{m_k}=x_0^{(n_{k+1})}$. By concatenating obtained chains, we get the following limit pseudo orbit:
\begin{equation*}
\alpha=\gamma_{n_1}\beta_1\gamma_{n_2}\beta_2\gamma_{n_3}\beta_3\gamma_{n_4}\beta_4\cdots=(x_i)_{i\ge0}\subset S.
\end{equation*}
However, by the choice of $\gamma_n$, we have $\limsup_{i\to\infty}d(x_i,f^i(y))\ge\epsilon$ for every $y\in X$, which contradicts that
$f$ has the limit shadowing property around $S$. Thus, $f$ has the shadowing property around $S$.
\end{proof}

The next lemma is a modification of \cite[Theorem 3.4.2]{AH} and \cite[Lemma 1]{M} which state that if $f$ has the shadowing property, then so does $f|_{\Omega(f)}$. Although the proof is similar to that of \cite[Theorem 3.4.2]{AH} or \cite[Lemma 1]{M}, we shall give it for completeness. Note that we consider $\overline{M(f)}$ instead of $\Omega(f)$ and only assume the shadowing property around $\overline{M(f)}$ (not the global shadowing property).
 
\begin{lem}
If a continuous map $f\colon X\to X$ has the shadowing property around $\overline{M(f)}$, then $f|_{\overline{M(f)}}\colon\overline{M(f)}\to\overline{M(f)}$ satisfies the shadowing property.
\end{lem}

\begin{proof}
Given any $\epsilon>0$, take $\delta>0$ so small that every $\delta$-pseudo orbit of $f$ in $\overline{M(f)}$ is $\epsilon$-shadowed by some point of $X$. Note that $CR(f|_{\overline{M(f)}})=\overline{M(f)}$ and let ${\mathcal C}(\delta)$ be the set of $\delta$-chain components of $\overline{M(f)}$ (see the proof of Lemma 2.1). Fix $0<\eta<\delta$ and let $\alpha=(x_i)_{i\ge0}\subset\overline{M(f)}$ be an $\eta$-pseudo orbit of $f$. We shall show that $\alpha$ is $\epsilon$-shadowed by some $x\in\overline{M(f)}$. As the above proof, $\alpha$ is contained in some $C\in{\mathcal C}(\delta)$. For each $n\ge1$, since $\{x_0,x_n\}\subset C$, there is a $\delta$-chain $\beta_n=(x_i^{(n)})_{i=0}^{k_n}\subset\overline{M(f)}$ of $f$ such that $x_0^{(n)}=x_n$ and $x_{k_n}^{(n)}=x_0$. Put $\alpha_n=(x_0,x_1,\dots,x_n)$, $m_n=n+k_n$, and consider the $m_n$-periodic $\delta$-pseudo orbit \[\gamma_n=\alpha_n\beta_n\alpha_n\beta_n\cdots\subset\overline{M(f)}\] of $f$, which is $\epsilon$-shadowed by some $y_n\in X$. Then, $\gamma_n$ is $\epsilon$-shadowed by $f^{m_na}(y_n)$ for every $a\ge0$, so is by any $y\in\overline{O_{f^{m_n}}(y_n)}$. Since $\overline{O_{f^{m_n}}(y_n)}$ is a compact $f^{m_n}$-invariant subset, we have $\overline{O_{f^{m_n}}(y_n)}\cap M(f^{m_n})\ne\emptyset$; Zorn's lemma implies the existence of at least one minimal point. Fix an \[x_n\in \overline{O_{f^{m_n}}(y_n)}\cap M(f^{m_n}).\] Then, $\gamma_n$ is $\epsilon$-shadowed by $x_n$, and $x_n\in M(f^{m_n})=M(f)$ (see Chapter V of \cite{BC} for a proof of this equality). Now, $\alpha_n$ is $\epsilon$-shadowed by $x_n\in M(f)$ for each $n\ge1$. Take $x\in X$ and an increasing sequence of integers $0<n_1<n_2<$ with $\lim_{j\to\infty}x_{n_j}=x$. Then, we have $x\in\overline{M(f)}$ and easily see that $\alpha$ is $\epsilon$-shadowed by $x$, completing the proof. 
\end{proof}

Now, let us prove Theorem 1.1.

\begin{proof}[Proof of Theorem 1.1]
It is obvious from the definition that $f$ has the limit shadowing property around $\overline{M(f)}$. Since $CR(f|_{\overline{M(f)}})=\overline{M(f)}$, Lemma 2.1 implies that $f$ has the shadowing property around $\overline{M(f)}$. Then, by Lemma 2.2, \[f|_{\overline{M(f)}}\colon\overline{M(f)}\to\overline{M(f)}\] satisfies the shadowing property; therefore, it only remains to prove $CR(f)=\Omega(f)=\overline{M(f)}$. For the purpose, it suffices to prove $CR(f)\subset\overline{M(f)}$, because $\overline{M(f)}\subset\Omega(f)\subset CR(f)$ holds for every continuous map $f$.

Given any $x\in CR(f)$, take an $n^{-1}$-cycle $\gamma_n=(x_i^{(n)})_{i=0}^{m_n}$ of $f$ with $x_0^{(n)}=x_{m_n}^{(n)}=x$ for each integer $n\ge1$. By concatenating them, we obtain a limit pseudo orbit $\alpha=\gamma_1\gamma_2\gamma_3\cdots$ of $f$, which has a limit shadowing point $y\in X$. Then, it is clear that $x\in\omega(y)$. Note that $f|_{\omega(y)}$ satisfies $CR(f|_{\omega(y)})=f|_{\omega(y)}$, since $f|_{\omega(y)}$ is internally chain transitive (see \cite{BGOR} for details of the notion). Because $f$ has the limit shadowing property around $\omega(y)$, by Lemma 2.1, $f$ has the shadowing property around $\omega(y)$. Given any $\epsilon>0$, take $\delta>0$ so small that every $\delta$-pseudo orbit of $f$ in $\omega(y)$ is $\epsilon$-shadowed by some point of $X$. Then, again by $CR(f|_{\omega(y)})=f|_{\omega(y)}$, there is a $\delta$-cycle $\gamma=(x_i)_{i=0}^m\subset\omega(y)$ of $f$ with $x_0=x_m=x$. The rest of the proof is identical to that of \cite[Theorem 1]{M}. Consider the $m$-periodic $\delta$-pseudo orbit $\beta=\gamma\gamma\gamma\cdots\subset\omega(y)$ of $f$, which is $\epsilon$-shadowed by some $z\in X$. Note that $\overline{O_{f^m(z)}}$ is a compact $f^m$-invariant subset of $X$ contained in $B_\epsilon(x)=\{w\in X\colon d(x,w)\le\epsilon\}$. Fix a $w\in\overline{O_{f^m(z)}}\cap M(f^m)$. Then, it follows that  $d(x,w)\le\epsilon$ and $w\in M(f)$. Since $x\in CR(f)$ and $\epsilon>0$ are arbitrary, we conclude $CR(f)\subset\overline{M(f)}$.
\end{proof}

\section{Proof of Theorem 1.2 and Corollary 1.3}

In this section, we prove Theorem 1.2 and Corollary 1.3. First, we prove the following lemma.

\begin{lem}
Let $f\colon X\to X$ be an equicontinuous homeomorphism. Then, the following three properties are equivalent
\begin{itemize}
\item[(1)] $f$ has the limit shadowing property,
\item[(2)] $f$ has the shadowing property,
\item[(3)] ${\rm dim}\,X=0$, or equivalently, $X$ is totally disconnected.
\end{itemize} 
\end{lem}

A remark is needed before the proof of Lemma 3.1. The implication $(1)\Rightarrow(2)$ in Lemma 3.1 is an immediate corollary of Theorem 1.1. We need only note that $X=M(f)=\Omega(f)$ holds for every equicontinuous homeomorphism $f\colon X\to X$, as mentioned in Section 1. The equivalence $(2)\Leftrightarrow (3)$ is an already known fact on the shadowing property (see \cite[Theorem 4]{M}). The implication $(3)\Rightarrow(1)$ can be verified by a simple combination of the results given in \cite{BGO}, which are (A) Every equicontinuous homeomorphism $f\colon X\to X$ with ${\rm dim}\,X=0$ satisfies the h-shadowing property, and (B) h-shadowing property implies the s-limit shadowing property, and then the limit shadowing property. However, we provide a self-contained proof of $(3)\Rightarrow(1)$ below. For the purpose, we need the next lemma.

\begin{lem}
Let $f\colon X\to X$ be an equicontinuous map. If ${\rm dim}\,X=0$, then for every $\epsilon>0$, there exists $\delta=\delta(\epsilon)>0$ such that any $\delta$-pseudo orbit $(x_i)_{i\ge0}$ of $f$ is $\epsilon$-shadowed by $x_0$, that is, $d(x_i,f^i(x_0))\le\epsilon$ for every $i\ge0$. 
\end{lem}

\begin{proof}
We can assume that $d(f(x),f(y))\le d(x,y)$ for all $x,y\in X$. Since ${\rm dim}\,X=0$, for any given $\epsilon>0$, we can find a decomposition of $X$ into a disjoint union of clopen subsets $A_1,\dots,A_m\subset X$ such that ${\rm diam}\,A_k\le\epsilon$ for every $1\le k\le m$. Take $0<\delta<\min_{1\le k<l\le m}d(A_k,A_l)$ and suppose that  $(x_i)_{i\ge0}$ is a $\delta$-pseudo orbit  of $f$. We prove by induction on $i$ that the following property holds for every $i\ge 1$. 
\begin{equation*}
\text{$d(f^{i-j}(x_j),f^{i-j-1}(x_{j+1}))\le\delta$,\quad$0\le j\le i-1$} \tag{$P_i$}.
\end{equation*}
When $i=1$, $(P_1)$ is just $d(f(x_0),x_1)\le\delta$, which is obviously true. Suppose that $(P_i)$ holds for some $i\ge1$. Then, we have
\begin{equation*}
d(f^{i+1-j}(x_j),f^{i-j}(x_{j+1}))\le d(f^{i-j}(x_j),f^{i-j-1}(x_{j+1}))\le\delta,\quad0\le j\le i-1 
\end{equation*}
and $d(f(x_i),x_{i+1})\le\delta$, which implies that $(P_{i+1})$ holds, and so completes the induction. Now, for every $i\ge1$, from $(P_i)$ and the choice of $\delta$, it follows that \[\{f^{i-j}(x_j)\colon0\le j\le i\}\subset A_{k(i)}\] for some $1\le k(i)\le m$, especially $f^i(x_0),x_i\in A_{k(i)}$. Since ${\rm diam}\,A_{k(i)}\le\epsilon$, we have $d(x_i,f^i(x_0))\le\epsilon$ for each $i\ge1$, and this proves the lemma. 
\end{proof}

Then, let us prove Lemma 3.1.

\begin{proof}[Proof of Lemma 3.1]
As remarked above it remains to show that $(3)\Rightarrow(1)$. We can assume that $d(f(x),f(y))=d(x,y)$ for all $x,y\in X$. Let us suppose that $(x_i)_{i\ge0}$ is a limit pseudo orbit of $f$ and prove that  $(x_i)_{i\ge0}$ has a limit shadowing point. Since $\lim_{i\to\infty}d(f(x_i),x_{i+1})=0$, by Lemma 3.2, we can take a sequence of integers $0\le i_1<i_2<\cdots$ such that for each $n\ge1$, we have 
\begin{equation*}
d(x_{i_n+j},f^j(x_{i_n}))\le2^{-n},\:\:\text{for every}\:\:j\ge0 \tag{a}.
\end{equation*}
Put $y_n=f^{-i_n}(x_{i_n})$ for each $n\ge1$. Then, we have
\begin{eqnarray*}
d(y_n,y_{n+1})&=&d(f^{-i_n}(x_{i_n}),f^{-i_{n+1}}(x_{i_{n+1}})) \\
                 &=&d(f^{i_{n+1}-i_n}(x_{i_n}),x_{i_{n+1}})\\
                 &\le&2^{-n}\quad\text{(by taking $j=i_{n+1}-i_n$ in (a))} 
\end{eqnarray*}
for every $n\ge1$. This implies that $(y_n)_{n\ge1}$ is a Cauchy sequence of points in $X$, and thus there is $y\in X$ such that $\lim_{n\to\infty}y_n=y$. We have
\begin{equation*}
d(y_n,y)\le2^{-n}+2^{-n-1}+2^{-n-2}+\cdots=2^{-n+1} \tag{b}
\end{equation*}
for any $n\ge1$. It follows that
\begin{eqnarray*}
d(x_{i_n+j},f^{i_n+j}(y))&\le&d(x_{i_n+j},f^{i_n+j}(y_n))+d(f^{i_n+j}(y_n),f^{i_n+j}(y)) \\
                           &=&d(x_{i_n+j},f^j(x_{i_n}))+d(y_n,y) \\
                           &\le&2^{-n}+2^{-n+1}=3\cdot2^{-n}\quad\text{(by (a) and (b))}
\end{eqnarray*}
for all $n\ge1$ and $j\ge0$. Thus, we have $\lim_{i\to\infty}d(x_i,f^i(y))=0$, proving the lemma.
\end{proof}

A few more simple lemmas are needed for the proof of Theorem 1.2. By analogy with the notion of an $\omega$-limit set of a point in a dynamical systems, we define an $\omega$-limit set for any sequence of points in $X$. For a sequence of points $\gamma=(x_i)_{i\ge0}$ in $X$, we denote by $\omega(\gamma)$ the set of points $x\in X$ for which there is an increasing sequence of integers $0\le i_1<i_2<\cdots$ such that $\lim_{j\to\infty}x_{i_j}=x$. By the compactness of $X$, we can easily see that $\lim_{i\to\infty}d(x_i,\omega(\gamma))=0$ for any $\gamma=(x_i)_{i\ge0}$.

\begin{lem}
Let $f\colon X\to X$ be a continuous map and let $\gamma=(x_i)_{i\ge0}$ be a limit pseudo orbit of $f$. Then, we have $\omega(\gamma)\subset CR(f)$ and especially $\lim_{i\to\infty}d(x_i,CR(f))=0$.
\end{lem}

\begin{proof}
Let $x\in\omega(\gamma)$ and let us show that $x\in CR(f)$. Given $\delta>0$, take $0<\epsilon\le\delta/2$ such that $d(a,b)\le\epsilon$ implies $d(f(a),f(b))\le\delta/2$ for all $a,b\in X$. Since $\gamma$ is a limit pseudo orbit of $f$, there is an integer $n>0$ such that $d(f(x_i),x_{i+1})\le\delta/2$ for any $i\ge n$. For such $n$, take $J>I\ge n$ so large that $d(x_I,x)\le\epsilon$ and $d(x_J,x)\le\epsilon$. Then, we consider a cycle $\beta=(x,x_{I+1},x_{I+2},\dots,x_{J-1},x)$. By the choice of $\epsilon$, $n$, $I$, and $J$, we have
\begin{equation*}
d(f(x_i),x_{i+1})\le\delta/2\le\delta
\end{equation*}
for every $I+1\le i\le J-2$,
\begin{equation*}
d(f(x),x_{I+1})\le d(f(x),f(x_I))+d(f(x_I),x_{I+1})\le\delta/2+\delta/2\le\delta,
\end{equation*}
and
\begin{equation*}
d(f(x_{J-1}),x)\le d(f(x_{J-1}),x_J)+d(x_J,x)\le\delta/2+\epsilon\le\delta,
\end{equation*}
implying that $\beta$ is a $\delta$-cycle of $f$. Since $\delta>0$ is arbitrary, we conclude that $x\in CR(f)$.
\end{proof} 

\begin{lem}
Let $f\colon X\to X$ be a continuous map. If $f$ has the limit shadowing property around $CR(f)$, then $f$ has the limit shadowing property.
\end{lem}

\begin{proof}
Let $\gamma=(x_i)_{i\ge0}$ be a limit pseudo orbit of $f$. By Lemma 3.3, we have
\begin{equation*}
\lim_{i\to\infty}d(x_i,CR(f))=0.
\end{equation*}
For each $i\ge0$, take $y_i\in CR(f)$ such that $d(x_i,y_i)=d(x_i,CR(f))$ (we can do that since $CR(f)$ is closed). Then, since
\begin{equation*}
d(f(y_i),y_{i+1})\le d(f(y_i),f(x_i))+d(f(x_i),x_{i+1})+d(x_{i+1},y_{i+1})
\end{equation*}
for every $i\ge0$, and each term of the right-hand side tends to $0$ as $i$ tends to $\infty$, we have $\lim_{i\to\infty}d(f(y_i),y_{i+1})=0$, i.e., $(y_i)_{i\ge0}\subset CR(f)$ is a limit pseudo orbit of $f$. Hence, there exists $y\in X$ such that $\lim_{i\to\infty}d(y_i,f^i(y))=0$, and then by
\begin{equation*}
d(x_i,f^i(y))\le d(x_i,y_i)+d(y_i,f^i(y)),
\end{equation*}
we see that $\lim_{i\to\infty}d(x_i,f^i(y))=0$. Thus, $y$ gives a desired limit shadowing point of $\gamma$.
\end{proof} 

\begin{rem}
\normalfont
Given a continuous map $f\colon X\to X$, it is obvious from the definition that if $f|_{CR(f)}$ has the limit shadowing property, then $f$ has the limit shadowing property around $CR(f)$, so by Lemma 3.4, $f$ has the limit shadowing property.
\end{rem}

The next lemma on equicontinuous maps is proved in \cite{Ma}.

\begin{lem}\cite{Ma}
Let $f\colon X\to X$ be an equicontinuous map. Then, the following properties hold
\begin{itemize}
\item[(1)] $\Omega(f)=CR(f)=\bigcap_{n\ge1}f^n(X)$,
\item[(2)] $f|_{\Omega(f)}$ is an equicontinuous homeomorphism.
\end{itemize} 
\end{lem}

Now, we give a proof of Theorem 1.2.

\begin{proof}[Proof of Theorem 1.2]
$(2)\Rightarrow(1)$: It is well-known that if $f$ has the shadowing property, then so does $f|_{\Omega(f)}$ (see \cite[Lemma 1]{M}). Since $f|_{\Omega(f)}$ is an equicontinuous homeomorphism by Lemma 3.5 (2), we use Lemma 3.1 to conclude that $f|_{\Omega(f)}$ has the limit shadowing property. Note that we have $CR(f)=\Omega(f)$ as a consequence of the shadowing property of $f$ or Lemma 3.5 (1). Thus, by Lemma 3.4, $f$ has the limit shadowing property.  
$ $\newline
$(1)\Rightarrow(3)$: By Theorem 1.1, if $f$ has the limit shadowing property, then $f|_{\Omega(f)}$ has the shadowing property. Since $f|_{\Omega(f)}$ is an equicontinuous homeomorphism by Lemma 3.5 (2), we have that $\dim\,\Omega(f)=0$ by Lemma 3.1.
$ $\newline
$(3)\Rightarrow(2)$: It suffices to prove that every $x\in X$ is a chain continuity point for $f$. Here, for any given $x\in X$, following \cite{A}, we define $C(x)\subset X$ by
\begin{equation*}
C(x)=\{y\in X\colon\text{$\forall\delta>0$ $\forall n\ge 1$ $\exists$ $\delta$-chain $(x_i)_{i=0}^k$ of $f$ s.t. $k\ge n$, $x_0=x$ and $x_k=y$}\}.
\end{equation*}
Since $f$ is equicontinuous, by \cite{A}, if $\dim\,C(x)=0$, then $x$ is a chain continuity point for $f$. Note that we have $C(x)\subset\bigcap_{n\ge1}f^n(X)$ and $\bigcap_{n\ge1}f^n(X)=\Omega(f)$ by Lemma 3.5 (1). Thus, $\dim\,\Omega(f)=0$ implies $\dim\,C(x)=0$, and so every $x\in X$ is a chain continuity point for $f$.
\end{proof}

As the final proof of this section, we prove Corollary 1.3.

\begin{proof}[Proof of Corollary 1.3]
Item (1) is an immediate consequence of Theorem 1.1 and Proposition 1.1. As for (2), by Corollary 1.2, the limit shadowing property implies the thick shadowing property for every continuous map. Conversely, suppose that a c-expansive or equicontinuous map $f\colon X\to X$ satisfies the thick shadowing property. Then, by \cite[Theorem 4.5]{O2}, $f|_{CR(f)}$ has the shadowing property. From Proposition 1.1,  it follows that $f|_{CR(f)}$ has the limit shadowing property, which combined with Lemma 3.4 proves that $f$ has the limit shadowing property.
\end{proof}

\section{Example: a complement to Theorem 1.2}

In this section, we give an example of equicontinuous map showing that the s-limit shadowing property is not necessarily equivalent to the three properties in Theorem 1.2. First, we recall the definition of the s-limit shadowing property. A continuous map $f\colon X\to X$ is said to have the {\em s-limit shadowing property} if for any $\epsilon>0$, there exists $\delta>0$ satisfying the following properties
\begin{itemize}
\item[(1)] Every $\delta$-pseudo orbit of $f$ is $\epsilon$-shadowed by some point of $X$,
\item[(2)] For every $\delta$-pseudo orbit $(x_i)_{i\ge 0}$ which is also a limit pseudo orbit of $f$, there is an $\epsilon$-shadowing point $x\in X$ which is also a limit shadowing point of  $(x_i)_{i\ge 0}$. 
\end{itemize}

Our example is a simple modification of an odometer, so we shall recall its definition. An {\em odometer} (also called an {\em adding machine}) is defined as follows. Given a strictly increasing sequence $m=(m_k)_{k\ge1}$ of positive integers such that $m_1\ge2$ and $m_k$ divides $m_{k+1}$ for each $k=1,2,\dots$, we define
\begin{itemize}
\item $X(k)=\{0,1,\dots,m_k-1\}$ (with the discrete topology),
\item $X_m=\{(x_k)_{k\ge1}\in\prod_{k\ge1}X(k)\colon x_k\equiv x_{k+1}\pmod{m_k}\colon\forall k\ge1\}$,
\item $g(x)_k=x_k+1\pmod{m_k}$ for all $x=(x_k)_{k\ge1}\in X_m$, $k\ge1$.
\end{itemize}
The set $X_m$ has the subspace topology induced by the product topology on $\prod_{k\ge1}X(k)$, and the resulting dynamical system $(X_m,g)$ is called an odometer with the periodic structure $m$. It is immediate from the definition that $g\colon X_m\to X_m$ is an equicontinuous homeomorphism, and in fact, the odometers are characterized as the minimal equicontinuous systems on Cantor spaces (see \cite{Ku}).

\begin{ex}
\normalfont
Let $g\colon X_m\to X_m$ be an odometer. We define a metric $d$ on $X_m$ by
\begin{equation*}
d(x,y)=\sup_{k\ge1}2^{-k}\delta(x_k,y_k)
\end{equation*}
for $x=(x_k)_{k\ge1}, y=(y_k)_{k\ge1}\in X_m$, where $\delta(a,b)=0$ if $a=b$ and $\delta(a,b)=1$ otherwise. Then, let $X=\{p\}\sqcup X_m$ be the disjoint union of a one point set $\{p\}$ and $X_m$. We can extend the metric $d$ to on $X$ and assume that $d(p,x)>1$ for every $x\in X_m$. Put $q=(0,0,\dots)\in X_m$ and define $f\colon X\to X$ by $f(p)=q$ and $f(x)=g(x)$ for every $x\in X_m$. It is clear that $f$ is equicontinuous since $g$ is so, and ${\rm dim}\,\Omega(f)={\rm dim}\,X_m=0$, hence by Theorem 1.2, $f$ has the shadowing property. Let us show that $f$ does not satisfy the s-limit shadowing property. For the purpose, it is sufficient to prove that for any $\delta>0$, there is a $\delta$-pseudo and limit pseudo orbit $\gamma=\gamma(\delta)$ of $f$ such that every 1-shadowing point of it is not a limit shadowing point, so we shall construct such $\gamma$.
  
First, for any given $\delta>0$, we take an integer $K\ge2$ with $2^{-K}\le\delta$ and put
\begin{equation*}
C_j=\{(x_k)_{k\ge1}\in X_m\colon x_K=j\}\subset X_m
\end{equation*}
for each $j\in\{0,1,\dots,m_K-1\}$. Note that $q\in C_0$. It is clear that $X_m=\bigsqcup_{j=0}^{m_K-1}C_j$ is a clopen partition of $X_m$, and we have $f(C_j)=C_{j+1}\pmod{m_K}$ for every $j\in\{0,1,\dots,m_K-1\}$. Put $r=f^{m_{K-1}}(q)=(0,0,\dots,0,m_{K-1},m_{K-1},\dots)\in C_{m_{K-1}}$ and note that $d(f(p),r)=d(q,r)=2^{-K}\le\delta$. Then, the following pseudo orbit: 
\begin{equation*}
\gamma=(x_i)_{i\ge0}=(p,r,f(r),f^2(r),f^3(r),\dots),
\end{equation*}
is a $\delta$-pseudo and limit pseudo orbit of $f$. Now, suppose that $x\in X$ is a $1$-shadowing point of $\gamma$. Since $d(p,x)=d(x_0,x)\le1$, by the property of the metric $d$, we have $x=p$. However, for every integer $j\ge1$, we have
\begin{equation*}
x_j=f^{j-1}(r)\in C_{m_{K-1}+j-1}\pmod{m_K}
\end{equation*}
and
\begin{equation*}
f^j(p)=f^{j-1}(q)\in C_{j-1}\pmod{m_K}.
\end{equation*}
From $m_{K-1}+j-1\not\equiv j-1\pmod{m_K}$, it follows that
\begin{equation*}
d(x_j,f^j(p))\ge 2^{-K}
\end{equation*}
for every $j\ge1$, which implies that $p=x$ is not a limit shadowing point of $\gamma$. Thus, $f$ does not satisfy the s-limit shadowing property.  
\end{ex}

\section*{Acknowledgements}

The author would like to thank the referee for helpful comments and suggestions.

\end{document}